\documentclass{amsart}

\usepackage{latexsym,amsthm,amssymb}

\usepackage{amscd}

\DeclareMathAlphabet{\mathpzc}{OT1}{pzc}{m}{it}

\usepackage[leqno]{amsmath}

\usepackage{mathrsfs}

\usepackage[all]{xy}

\usepackage{color}
\usepackage{colordvi}

\usepackage{verbatim}
\usepackage{ifthen}
\newboolean{hide.proofs}
\setboolean{hide.proofs}{false}
\newtheorem{theorem}{Theorem}[section]

\newtheorem{definition}[theorem]{Definition}
\newtheorem{proposition}[theorem]{Proposition}

\newtheorem{remark}[theorem]{Remark}
\newtheorem{noname}[theorem]{}
\newtheorem{observation}[theorem]{Observation}

\newtheorem{lemma-conjecture}[theorem]{Lemma--Conjecture}

\newtheorem{corollary}[theorem]{Corollary}

\numberwithin{equation}{theorem}

\renewcommand{\mathcal}{\mathscr}

\newcommand{\Cal}{\mathcal}

\newcommand{\SE}{{\mathcal{E}}}

\newcommand{\SI}{{\mathcal{I}}}

\newcommand{\SO}{{\mathcal{O}}}

\newcommand{\SY}{{\mathcal{Y}}}

\renewcommand{\mathbb}{\mathbf}

\newcommand{\WU}{\widetilde{U}}

\newcommand{\WY}{\widetilde{Y}}

\title[Smoothable multiple
structures on curves]{Smoothable locally Cohen--Macaulay and non Cohen--Macaulay
multiple structures on curves}

\author{Francisco Javier Gallego}
\author{Miguel Gonz\'alez}
\author{\\ Bangere P. Purnaprajna}

\address{Departamento de \'Algebra and Instituto de Matem\'atica
Interdisciplinar, Universidad Complutense de Madrid}
\email{gallego@mat.ucm.es}
\address{Departamento de \'Algebra, Universidad Complutense de Madrid}
\email{mgonza@mat.ucm.es}
\address{Department of Mathematics, University of Kansas}
\email{purna@math.ku.edu}
\thanks{The first and the second author were partially supported by grants
MTM2006--04785 and MTM2009--06964 and by the UCM research group 910772. They
also thank the Department of Mathematics of the University of
Kansas for its hospitality. The third author thanks the General Research Fund
(GRF) of the University of Kansas  and the Simons foundation collaboration grant
for partially supporting this research. He
also thanks the Department of Algebra of the Universidad Complutense de Madrid
for
its hospitality.}

\begin{document}

\begin{abstract} In this article we show that a wide range of multiple
structures
on curves arise whenever a family of embeddings degenerates to a morphism
$\varphi$ of
degree $n$. One could expect to see, when an embedding
degenerates to such a morphism, the appearance of a
locally
Cohen--Macaulay multiple structure of certain kind (a so--called \emph{rope}
of multiplicity $n$). We show that
this expectation is naive and that both locally Cohen--Macaulay and non
Cohen--Macaulay multiple structures might occur in this situation. In seeing
this
we find out that many multiple structures can be \emph{smoothed}. When we specialize
to the case of double structures we are
able to say much more. In particular, we find numerical conditions, in terms
of the degree and the arithmetic genus, for the existence of many smoothable
double structures. Also, we show that the
existence of these double structures is determined, although not uniquely, by
the elements of certain space of vector bundle homomorphisms, which are related
to the first order infinitesimal deformations of $\varphi$.
In many instances, we show that, in order to determine a double structure
uniquely, looking merely at a first order deformation of $\varphi$ is not
enough; one needs to choose also a formal deformation.

\end{abstract}

\maketitle

\section*{Introduction}

\noindent In this article we study projective \emph{multiple structures}
embedded in projective space, not necessarily locally Cohen--Macaulay, on curves.
We then specialize our study to the case of \emph{double structures} to
obtain very concrete results. We do this following two motivations. Our
first motivation is to give evidence that multiplicity $n$ structures of a
certain kind  arise
naturally
when a family of embeddings degenerates to a morphism of degree $n$ or, more
precisely,
as the flat limit of the images of the embeddings of such families. Multiple
structures are everywhere non reduced schemes. Still, some of them, the
so--called \emph{ropes} are relatively nice (a rope on a smooth,
irreducible projective curve $Y$ of $\mathbf P^r$, is locally Cohen--Macaulay
and contained in the first infinitesimal neighborhood of $Y$).
When the multiple structure is a rope, and more so, when, in
addition, its multiplicity is $2$ (we call it in this case a \emph{ribbon})
this phenomenon
has been studied in several situations (see e.g. \cite{Fong},
\cite{GGPropes} and \cite{criterion}). A family of embeddings degenerating to
a finite morphism is a natural occurrence in algebraic geometry, and considering
the past evidence, it is reasonable to expect the natural limit of such a
degeneration to be a rope. Our results show that this
expectation is naive
and far from true and this is
one of the novelties of this article: to show
(see  Theorem~\ref{general.nonCMsmoothing} for multiplicity $n$
structures and
Corollary~\ref{general.nonCMsmoothing.ribbons},
Proposition~\ref{general.nonCMsmoothing.corollary} and
Theorems~\ref{elliptic} and \ref{elliptic2} for double structures) that a wide range of
locally non Cohen--Macaulay multiple structures of multiplicity $n$ that are
\emph{generically ropes} do appear naturally as limits of
images of embeddings that degenerate to an $n$--to--one morphism.
In
particular, we prove that all those multiple structures are \emph{smoothable}
(i.e., can be deformed to smooth, irreducible curves). Observe that in the case
$n=2$, all multiple structures, i.e., all double
structures, are generically ribbons. Thus when we focus on multiplicity $2$ in
Sections 2 to 5 we are looking at double structures in all generality.  Why is that
sometimes the
limit of images of a family of embeddings is a locally non Cohen--Macaulay
multiple
structure, instead  of a rope? Being locally non Cohen--Macaulay is,
obviously, a local question and indeed, the appearance of locally non
Cohen--Macaulay multiple structures as limits should be detected by local
computations. However, our results show that it is the
global geometry that dictates the local algebra and that the appearance of
the locally non Cohen--Macaulay structures is the direct result of such a
dictate. More precisely, this is the global reason for the appearance of
locally non
Cohen--Macaulay multiple structures:
the genus of the smooth, embedded curves do not
always match the arithmetic genus of the expected rope. If such is the case,
the
flat limit has to contain  necessarily not only this rope but also some embedded
points so
that the genera can match.   Theorems~\ref{general.nonCMsmoothing},
\ref{elliptic} and \ref{elliptic2},
Corollary~\ref{general.nonCMsmoothing.ribbons} and
Propositions~\ref{general.nonCMsmoothing.corollary} and \ref{nonCM.numerical.2}
show that this situation happens very often.

\medskip

\noindent Our second motivation is to give a geometric interpretation of the
non
surjective homomorphisms belonging to Hom$(\mathcal I/\mathcal I^2, \mathcal
E)$. Before going on, let us say what $\mathcal I$ and $\mathcal E$ are. Let
$X$ and $Y$ be smooth, irreducible projective curves, let $i$ be an embedding
of $Y$ in $\mathbf P^r$ and let $\mathcal I$ be the ideal sheaf of $i(Y)$ in
$\mathbf P^r$. Assume there exists a morphism $\varphi$ from $X$ to $\mathbf
P^r$ and a  morphism $\pi$ of degree $n$ from $X$ to $Y$ such that $\varphi$
factors as $i \circ \pi$ and let $\mathcal E$ be the trace zero module of
$\pi$ (which is a vector bundle of rank $n-1$). It is well known (see for
instance~\cite{HV} for $n=2$; the arguments
for arbitrary $n$ are essentially the same) that a surjective element
of Hom$(\mathcal I/\mathcal I^2, \mathcal E)$ corresponds to a rope of
multiplicity $n$ in
$\mathbf P^r$ whose reduced structure is $i(Y)$. In this direction, the second
author generalized this fact in~\cite[Proposition 2.1]{Gon} and showed that an
element of
Hom$(\mathcal I/\mathcal I^2, \mathcal E)$ corresponds to a pair $(\widehat Y,
\hat i)$, where $\widehat Y$ is a rope on $Y$ and $\hat i$ is a morphism from
$\widehat Y$ to $\mathbf P^r$ extending $i$. There is another geometric
interpretation of the surjective elements $\mu$ of Hom$(\mathcal I/\mathcal I,
\mathcal E)$. In~\cite{criterion} we showed that, in a very general setting and
under natural conditions, such elements mean not only the existence of the
embedded ropes mentioned before, but also the fact that these ropes can be
smoothed. Thus, a natural question to ask is whether a similar picture exists
when $\mu$ is not surjective. In this article we show (see
Theorem~\ref{general.nonCMsmoothing}) that, under quite general and natural
conditions on $\mathcal E$ and $i(Y)$, the nonzero non surjective
homomorphisms
$\mu \in$  Hom$(\mathcal I/\mathcal I^2, \mathcal E)$ with image of rank $n-1$
bear witness to the
existence of locally non Cohen--Macaulay, rope--like multiple structures on
$Y$ which can be
smoothed. Moreover, if $n=2$, Corollary~\ref{general.nonCMsmoothing.ribbons}
shows that the nonzero non surjective
homomorphisms
$\mu \in$  Hom$(\mathcal I/\mathcal I^2, \mathcal E)$
bear witness to the
existence of locally non Cohen--Macaulay double structures on $Y$ which can be
smoothed.

\medskip

\noindent When we specialize to multiplicity $2$ structures, our geometric study
of the elements of \linebreak
Hom$(\mathcal I/\mathcal I^2,
\mathcal E)$
goes further and reveals that the assignment to $\mu$ of a smoothable double
structure is quite subtle.  We suggested in the previous paragraph that,
according to Corollary~\ref{general.nonCMsmoothing.ribbons}, $\mu$ ``produces" a
smoothable
double structure. However, as  Proposition~\ref{geom.inter.rat} shows, $\mu$ can
be related to many different double structures. To understand the process we
mention first that  the connection of an algebraic object such as $\mu$ (a
vector bundle homomorphism) with geometric objects such as deformations of
morphisms and double structures is made through the study of the first order
infinitesimal deformations of $\varphi$ (for details on this connection see
\cite[\S 3]{Gon}). Indeed, to each $\mu$ of Hom$(\mathcal I/\mathcal I^2,
\mathcal
E)$ we can associate (not in a unique way) a first infinitesimal deformation of
$\varphi$. Proposition~\ref{geom.inter.rat} tells that neither the information
encoded by $\mu$ nor even the first order infinitesimal deformation of $\varphi$
chosen are enough to determine a unique double structure. Instead, in order to
determine a unique double structure  we need to look not only at a first order
infinitesimal deformation of $\varphi$ but also at how this first order
infinitesimal deformation extends to higher orders. More geometrically put, in
order to determine a double structure starting from a given tangent vector $v$
to the base $\mathcal V$ of the semiuniversal deformation space of $\varphi$ we
need to choose a way to integrate $v$ along a path of $\mathcal V$.

\section{Smoothable rope--like multiple structures on
curves}\label{general.section}

\noindent The purpose of this section is to show that under quite general
conditions there exist smoothable multiple structures (which are
generically ropes but not necessarily locally
Cohen--Macaulay) on curves. We do so
in Theorem~\ref{general.nonCMsmoothing}, where we see how
these multiple structures appear naturally when we make degenerate a family of
embeddings to an $n$--to--one morphism.  First we set up
the notation to be used and recall the definitions of the
objects and concepts we will study in this section.

\begin{noname}\label{notation}
{\bf Notation and set--up:} {\rm Throughout the article, unless otherwise
explicitly stated, we will use the following
notation
and
set--up:
\begin{enumerate}
\item We work over an algebraically closed field of characteristic $0$.
\item $Y$ is a smooth, irreducible projective curve embedded in
$\mathbf P^r$, $r \geq 3$, of genus $g$ and degree $d$.
\item $\mathcal I$ is the ideal sheaf of
$Y$ in $\mathbf P^r$.
\item $\mathcal E$ is a vector bundle on $Y$ of rank $n-1$ and degree $-e$,
and $\tilde g=g - \chi(\mathcal E)$.
\item $\mu$ is a homomorphism of $\textrm{ Hom}(\mathcal I/\mathcal I^2,\mathcal
E)$,
 $\mathcal E'$ is the image of $\mu$
and $\hat g=g - \chi(\mathcal E')$.
\end{enumerate}}
\end{noname}

\begin{definition}\label{def.double.structure}
{\rm A subscheme
$\widetilde Y$ of $\mathbf P^r$ is a \emph{multiple structure} on  $Y$ of
multiplicity $n$
if the reduced structure of $\widetilde Y$ is $Y$ and the degree of $\widetilde
Y$ is $n$ times the degree of
$Y$. If $n=2$ we say that $\widetilde Y$ is a \emph{double structure} on $Y$.}
\end{definition}

\noindent Note that a multiple structure $\WY$ need not be locally
Cohen--Macaulay. Among locally Cohen--Macaulay multiple structures we are
interested in those ones called \emph{ropes} and among locally non
Cohen--Macaulay multiple structures we are interested in those who are
generically ropes. We give now the precise definitions:

\begin{definition}\label{def.ribbon}
{\rm Let $\WY$ be a scheme.
\begin{enumerate}
 \item We say that $\WY$ is a \emph{rope} of multiplicity $n$ on  $Y$ with
conormal bundle
$\SE$ (of rank $n-1$) if
the reduced structure of $\widetilde Y$ is $Y$ and
\begin{enumerate}
\item
$\SI^2=0$ and
\item
$\SI$ and $\SE$ are isomorphic as $\SO_{Y}$--modules.
\end{enumerate}
Note that a rope of multiplicity $n$ is a multiple structure of multiplicity
$n$. Note also that we can gave the same definition of rope if $Y$ is
a smooth curve non necessarily complete.
\item A rope of multiplicity $2$ on $Y$ is called a \emph{ribbon} on $Y$.
\item We say that $\WY$ is a \emph{rope--like} multiple
structure of
multiplicity $n$ on $Y$ if $\WY$ is generically a rope of
multiplicity
$n$ on $Y$, i.e., if there exists a non--empty open set $\WU$ of $\WY$
such that $\WU$ is a rope of
multiplicity $n$ on $\WU \cap Y$.
\end{enumerate}
}
\end{definition}

\begin{definition}\label{defi.smoothing}
 {\rm Let $\WY$ be a subscheme of $\mathbf P^r$. By a \emph{smoothing} of $\WY$
in
$\mathbf
P^r$ we mean a flat, integral family $\SY$ of subschemes of $\mathbf
P^r$ over a smooth affine, irreducible curve $T$, such that over a closed point
$0 \in T$,
$\SY_0=\WY$ and over the remaining points $t$ of $T$, $\SY_t$ is a smooth,
irreducible curve.}
\end{definition}

\noindent We introduce a homomorphism defined in~\cite[Proposition 3.7]{Gon}:

\begin{proposition}\label{morphism.miguel}
Let $C$ be a smooth irreducible projective curve, let $\pi: C  \longrightarrow
Y$ be a morphism of degree $n$,
let $\mathcal E$ be the trace zero module
of
$\pi$  and let $\varphi: C \longrightarrow \mathbf P^r$
be the composition of $\pi$ followed by the inclusion of $Y$ in $\mathbf P^r$.
There exists a homomorphism
\begin{equation*}
 H^0(\mathcal N_\varphi) \overset{\Psi}\longrightarrow \mathrm{Hom}(\pi^*(\mathcal I/\mathcal I^2), \mathcal O_X),
\end{equation*}
that appears when taking cohomology on the commutative diagram~\cite[(3.3.2)]{Gon}. Since
\begin{equation*}
\mathrm{Hom}(\pi^*(\mathcal I/\mathcal I^2), \mathcal O_X)=
\mathrm{Hom}(\mathcal I/\mathcal I^2, \pi_*\mathcal O_X)=\mathrm{Hom}(\mathcal
I/\mathcal I^2, \mathcal O_Y) \oplus \mathrm{Hom}(\mathcal I/\mathcal I^2,
\mathcal E),
\end{equation*}
the homomorphism $\Psi$ has two components,
\begin{eqnarray*}
H^0(\mathcal N_\varphi) & \overset{\Psi_1}  \longrightarrow &
\mathrm{Hom}(\mathcal I/\mathcal I^2, \mathcal O_Y) \textrm{ and } \cr
H^0(\mathcal N_\varphi) & \overset{\Psi_2}  \longrightarrow & \mathrm{Hom}(\mathcal I/\mathcal I^2, \mathcal E).
\end{eqnarray*}
\end{proposition}

\noindent We also recall this sequence of the commutative
diagram~\cite[(3.3.2)]{Gon}):
\begin{equation}\label{sequence.miguel}
0 \longrightarrow \mathcal N_\pi \longrightarrow \mathcal N_\varphi
\longrightarrow
\pi^*\mathcal N_{Y,\mathbf P^N} \longrightarrow 0.
\end{equation}

\begin{theorem}\label{general.nonCMsmoothing}
Let $Y, r,  \mathcal E, \tilde g, \mu, \mathcal E'$ and $\hat g$ be as
in \ref{notation} and assume that the rank of the image $\mathcal E'$ of $\mu$
is $n-1$.
If
\begin{enumerate}
\item there exist a smooth irreducible projective curve $C$ and a
morphism $\pi: C \longrightarrow Y$ of degree $n$ whose trace zero
module is $\mathcal E$;
\item $h^0(\mathcal O_Y(1)) + h^0(\mathcal E \otimes \mathcal O_Y(1))  \geq
r+1$; and
\item $h^1(\mathcal O_Y(1))=h^1(\mathcal E \otimes \mathcal O_Y(1))=0$,
\end{enumerate}
then there
exist
\begin{enumerate}
\item[(i)]  a rope--like multiple structure $\widetilde Y$  supported
on $Y$, of multiplicity $n$ and arithmetic genus $\tilde g$,
embedded in $\mathbf P^r$;
\item[(ii)] a rope
$\widehat Y$ supported on $Y$,  of multiplicity $n$ and arithmetic genus $\hat
g$ and with conormal
bundle $\mathcal E'$, contained in $\widetilde Y$;
\item[(iii)]
a smooth irreducible algebraic curve $T$
with a distinguished closed point $0$, a flat family $\mathcal C$ over $T$ and
a
$T$--morphism
$\Phi: \mathcal C \longrightarrow \mathbf P^{\tilde g}_{T}$ such that
\begin{enumerate}
\item[(a)] $\mathcal C_t$ is a smooth irreducible curve of genus $\tilde g$,
\item[(b)] $\Phi_t$ is an embedding
for all $t \neq 0$, and
\item[(c)] $\Phi_0$ is of degree $n$ onto $Y$ and
$[\Phi(\mathcal
X)]_0=\widetilde Y$
\end{enumerate}
(in particular, $\widetilde Y$ is smoothable in $\mathbf
P^r$).
\end{enumerate}
\end{theorem}

\begin{proof}
Let $\varphi$ be the composition of the
inclusion of $Y$ in $\mathbf P^r$ followed by $\pi$. Thus $C$ is a smooth curve
of genus $\tilde g$.
Since $C$ is a curve, $H^1(\mathcal N_\pi)=0$, so the map $\Psi_2$ introduced
in Proposition~\ref{morphism.miguel} surjects.
Let $\nu \in H^0(\mathcal
N_\varphi)$ be such that $\Psi_2(\nu)=\mu$. Such $\nu$ corresponds to an
infinitesimal deformation $\widetilde \varphi$ of $\varphi$.  Now we want to
apply~\cite[Theorem 1.1]{GGPropes}. By (3), $h^1(\mathcal O_Y(1))=h^1(\mathcal E
\otimes \mathcal O_Y(1))=0$ and by (2), $h^0(\varphi^*\mathcal O_{\mathbf
P^r}(1)) \geq r+1$. Then, according to~\cite[Theorem 1.1]{GGPropes} there exist
a smooth irreducible algebraic curve $T$
with a distinguished closed point $0$, a flat family $\mathcal C$ over $T$ and
a
$T$--morphism
$\Phi: \mathcal C \longrightarrow \mathbf P^{r}_{T}$ such that
\begin{enumerate}
\item[(a)] $\mathcal C_t$ is a smooth, irreducible curve;
\item[(b)] the restriction of $\Phi$ to the first infinitesimal neighborhood of
$0$
is $\tilde \varphi$ (and hence $\Phi_0=\varphi$); and
\item[(c)] for any $t \in T$, $t \neq 0$, $\Phi_t$ is an embedding into $\mathbf
P^r$.
\end{enumerate}
In particular, $\mathcal C_0=C$. Then, for all $t \in T$,
$\mathcal C_t$ is a smooth, irreducible curve of genus $\tilde g$, so $\mathcal
Y=\Phi(\mathcal C)$ is
a flat family over $T$ of $1$--dimensional subschemes of $\mathbf P^r$ of
arithmetic genus $\tilde g$.
Moreover, the reduced part of $\mathcal Y_0$ is $Y$ and, because of b) and c)
above,
the degree of $\mathcal Y_0$  is $n$ times the degree of $Y$. Thus $\mathcal
Y_0$ is a multiple structure of multiplicity $n$ and
arithmetic genus $\tilde g$, supported on $Y$. On the other hand, because of
b)
$\mathcal Y_0$ contains $(\textrm{im}\tilde{\varphi})_0$. By~\cite[Proposition
2.1]{Gon}), $\mu$ corresponds to a pair $(\overline Y,\overline i)$, where
$\overline Y$ is a rope of multiplicity $n$ on $Y$ and $\overline i$ a  morphism
(though not
necessarily a closed embedding), from $\overline Y$ to $\mathbf P^r$, that
extends the inclusion of $Y$ in $\mathbf P^r$. In addition, $\overline
i(\overline Y)$ is a rope of multiplicity $n$ whose conormal bundle is $\mathcal
E'$ and we set
$\widehat Y= \overline
i(\overline Y)$.
Now, by~\cite[Theorem 3.8 (1)]{Gon},
$(\textrm{im}\tilde{\varphi})_0$ equals $\overline i(\overline Y)$.
In particular, $\mathcal Y_0$ is a rope--like multiple structure
of multiplicity $n$ and
  arithmetic genus $\tilde g$.
Thus we may
set $\widetilde Y=\mathcal Y_0$.
\end{proof}

\begin{observation}\label{nonCMremark} {The following are equivalent:
\begin{enumerate}
\item $\mathcal E = \mathcal E'$ (i.e, $\mu$ is surjective),
\item $\tilde g=\hat g$,
\item $\widetilde Y=\widehat Y$,
\item $\widetilde Y$ is a rope of multiplicity $n$.
\end{enumerate}}
\noindent If $\widetilde Y$ is not a rope, then
$\widetilde Y$ contains  $\widehat Y$ and
some embedded points.
\end{observation}

\begin{proof}
(1) and (2) are equivalent by the definition of $\tilde g$ and $\hat g$. Since
$\widehat
Y$
is contained in $\widetilde Y$ and both are multiple structures on $Y$ with the
same multiplicity and,
respectively, of arithmetic genus $\hat g$ and $\tilde g$, (2) implies (3).
On the other hand, (3) obviously implies (2). Since $\widehat Y$ is a rope of
multiplicity $n$,
(3) implies (4). Finally, if $\widetilde Y$ is a rope of multiplicity $n$, since
$\widehat Y$ is also a rope of multiplicity $n$ and is contained in
$\widetilde Y$, then $\widetilde Y=\widehat Y$, so (4) implies (3).
\end{proof}

\section{Conditions for the existence of smoothable double structures on
curves}\label{general.double.section}

\noindent In the remaining of the article we focus on double
structures in general. It is well--known that a double
structure on $Y$ is always rope--like (or better yet, ribbon--like), since a
locally Cohen--Macaulay double structure on $Y$ is necessarily a ribbon, so
we will be able to use Theorem~\ref{general.nonCMsmoothing}.
Thus in this section we  apply Theorem~\ref{general.nonCMsmoothing} to give
sufficient numerical conditions to
guarantee the existence of smoothable double structures, not necessarily locally
Cohen--Macaulay, on curves.  We
summarize these numerical conditions in Proposition~\ref{nonCM.numerical.2}. We
keep using
\ref{notation}, but we give an extra piece of convention to be used in
this section and in the remaining of the article.

\begin{noname}\label{notation2}
{\bf Notation and set--up:} {\rm From now on,  $\mathcal E$ is a line bundle on
$Y$.}
\end{noname}

\noindent First of all, we give this corollary  of
Theorem~\ref{general.nonCMsmoothing} for double structures in general, in
which the assumption on the rank imposed on $\mu$ is superfluous:

\begin{corollary}\label{general.nonCMsmoothing.ribbons}
Let $Y, r,  \mathcal E, \tilde g, \mu, \mathcal E'$ and $\hat g$ be as
in \ref{notation} and \ref{notation2} and assume $\mu \neq 0$.
If
\begin{enumerate}
\item $|\mathcal E^{-2}|$ possesses a smooth, effective, non empty divisor;
\item $h^0(\mathcal O_Y(1)) + h^0(\mathcal E \otimes \mathcal O_Y(1))  \geq
r+1$; and
\item $h^1(\mathcal O_Y(1))=h^1(\mathcal E \otimes \mathcal O_Y(1))=0$,
\end{enumerate}
then there
exist
\begin{enumerate}
\item[(i)]  a double structure $\widetilde Y$  supported
on $Y$ of arithmetic genus $\tilde g$,
embedded in $\mathbf P^r$;
\item[(ii)] a ribbon
$\widehat Y$, supported on $Y$  of arithmetic genus $\hat g$ and with conormal
bundle $\mathcal E'$, contained in $\widetilde Y$;
\item[(iii)]
a smooth irreducible algebraic curve $T$
with a distinguished closed point $0$, a flat family $\mathcal C$ over $T$ and
a
$T$--morphism
$\Phi: \mathcal C \longrightarrow \mathbf P^{\tilde g}_{T}$ such that
\begin{enumerate}
\item[(a)] $\mathcal C_t$ is a smooth irreducible curve of genus $\tilde g$,
\item[(b)] $\Phi_t$ is an embedding
for all $t \neq 0$, and
\item[(c)] $\Phi_0$ is of degree $2$ onto $Y$ and
$[\Phi(\mathcal
X)]_0=\widetilde Y$
\end{enumerate}
(in particular, $\widetilde Y$ is smoothable in $\mathbf
P^r$).
\end{enumerate}
\end{corollary}

\begin{proof}
We just need to apply Theorem~\ref{general.nonCMsmoothing} for $n=2$. First
note that in this case $\mathcal E$ is a line bundle, so if $\mu$ is nonzero,
the rank of the image of $\mu$ is $1$. Conditions (2) and (3) of
Theorem~\ref{general.nonCMsmoothing} and
Corollary~\ref{general.nonCMsmoothing.ribbons} are stated exactly in the same
way. Condition (1) of Corollary~\ref{general.nonCMsmoothing.ribbons} implies
Condition (1) Theorem~\ref{general.nonCMsmoothing}, since Condition (1) of
Corollary~\ref{general.nonCMsmoothing.ribbons} implies the existence of
a double
cover $\pi: C \longrightarrow Y$ branched along a smooth, effective, non empty
divisor
$B$ in $|\mathcal E^{-2}|$. Then $C$ is smooth and irreducible and $\mathcal E$
 is the trace zero module of $\pi$. Finally, as already pointed out, locally
Cohen--Macaulay double structures are ribbons, so double
structures are always ribbon--like and, in (i) of
Theorem~\ref{general.nonCMsmoothing} we can substitute ``rope--like multiple
structure'' for ``ribbon'' and, obviously, in (ii) and (iii) of
Theorem~\ref{general.nonCMsmoothing}, ``rope'' for ``ribbon'' and ``degree $n$''
for ``degree $2$''.
\end{proof}

\noindent
In the next two
remarks we give simple
numerical conditions which
guarantee the existence of nonzero surjective and non
surjective homomorphisms $\mu$ of Hom$(\mathcal I/\mathcal I^2,\mathcal
E)$.

\begin{remark}\label{general.nonCMsmoothing.remark}
Assume that
(1) and (3)
of Corollary~\ref{general.nonCMsmoothing.ribbons} hold. In addition,
\begin{enumerate}
\item if $g=0$, let $3 \leq e \leq d -1$; and
\item if $g \geq 1$, let $d -e \geq g$,
\end{enumerate}
then there exists a nonzero, non surjective homomorphism $\mu$ of Hom$(\mathcal
I/\mathcal I^2,\mathcal
E)$.
\end{remark}

\begin{proof}
The existence
of a nonzero, non surjective homomorphism in
$\mathrm{Hom}(\mathcal I/\mathcal I^2, \mathcal E)$ is equivalent to
 the existence of a non zero global section in $H^0(\mathcal N_{Y, \mathbf
P^{r}} \otimes \mathcal E)$ with non empty vanishing locus. We consider the
diagram
\begin{equation}\label{N-F}
\xymatrix@C-7pt@R-10pt{
 & 0 \ar[d] & 0 \ar[d] && \\
  & \mathcal E \ar[d] \ar@{=}[r] & \mathcal E \ar[d] & & \\
0 \ar[r] & \mathcal F \otimes \mathcal E \ar[d] \ar[r] &
(\mathcal E \otimes \mathcal O_Y(1))^{\oplus r+1} \ar[d] \ar[r] &
\mathcal N_{Y, \mathbf P^{r}} \otimes \mathcal E \ar@{=}[d] \ar[r]& 0\\
0 \ar[r] & \mathcal T_{Y} \otimes \mathcal E \ar[d] \ar[r] & {\mathcal
T_{\mathbf P^{r }}}_{|Y} \otimes \mathcal E \ar[d] \ar[r] &  \mathcal N_{Y,
\mathbf P^{r}} \otimes \mathcal E \ar[r] &0\\
 & 0 & 0, &&}
\end{equation}
where $\mathcal F$ is the kernel of the composite surjective map
\begin{equation*}
H^0(\mathcal O_Y(1))^{\vee} \otimes \mathcal O_Y(1) \twoheadrightarrow {\mathcal
T_{\mathbf P^{r}}}_{|Y} \twoheadrightarrow \mathcal N_{Y, \mathbf P^{r}}.
\end{equation*}
Recall that $\mathcal E$ has no global sections.
We see also that $H^0(\mathcal T_{Y} \otimes \mathcal E)=0$.
Corollary~\ref{general.nonCMsmoothing.ribbons}, (1) implies $e \geq 0$, so, by
(1) the
degree of $\mathcal T_{Y} \otimes \mathcal E$  is negative except  if
$g=1$ and $e=0$. Since $\mathcal E$ is not trivial, in any case
$H^0(\mathcal T_{Y} \otimes \mathcal E)=0$. On the other hand,
Corollary~\ref{general.nonCMsmoothing}, (3)
 implies that $h^0(\mathcal E \otimes
\mathcal O_Y(1))=d-e-g + 1$, which is a positive number by (1) and
(2). Moreover, (1) and (2) also imply that the degree of $\mathcal E \otimes
\mathcal O_Y(1)$ is positive, so $\mathcal E \otimes \mathcal O_Y(1)$ has
non zero global sections with
non empty vanishing locus. Then taking global sections in the middle exact
sequence of \eqref{N-F} we obtain the desired non zero global sections of
$H^0(\mathcal N_{Y, \mathbf P^{r}} \otimes \mathcal
E)$ with
non empty vanishing locus.
\end{proof}

\begin{remark}\label{general.nonCMsmoothing.remark2}
Assume that
(1) and (3)
of Corollary~\ref{general.nonCMsmoothing.ribbons} hold. If
$\mathcal E \otimes \mathcal O_Y(1)$  is globally generated
(hence $1 \leq e \leq d$ when $g=0$ and $d - e \geq g + 1$
otherwise),
then Hom$(\mathcal I/\mathcal I^2,\mathcal E)$
possesses a surjective homomorphism.
\end{remark}

\begin{proof}
Since  $\mathcal E \otimes \mathcal O_Y(1)$ is globally generated, the result
follows from diagram~\eqref{N-F}.
\end{proof}

\noindent Corollary~\ref{general.nonCMsmoothing.ribbons} and
Remarks~\ref{general.nonCMsmoothing.remark}
and~\ref{general.nonCMsmoothing.remark2} yield the results that follow. The
proof of Proposition~\ref{general.nonCMsmoothing.corollary} is straight
forward.
Proposition~\ref{nonCM.numerical.2} summarizes sufficient conditions on the
genus
and degree of a linearly normal curve $Y$ of $\mathbf P^r$ to guarantee the
existence of smoothable double structures.

\begin{proposition}\label{general.nonCMsmoothing.corollary}
Assume that
(1), (2) and (3) of Corollary~\ref{general.nonCMsmoothing.ribbons} hold.
\begin{enumerate}
\item  If (1) or (2) of
Remark~\ref{general.nonCMsmoothing.remark} holds, then there exist a double
structure $\widetilde Y$ and a ribbon $\widehat Y$ satisfying (i), (ii) and
(iii) of
Corollary~\ref{general.nonCMsmoothing.ribbons}.
\item If $\mathcal E \otimes \mathcal O_Y(1)$ is globally generated, then
there
exist ribbons $\widetilde Y$ on $Y$ with conormal bundle $\mathcal E$ (hence, of
arithmetic genus $e+2g -1$) satisfying (i) and (iii) of
Corollary~\ref{general.nonCMsmoothing.ribbons}.
\end{enumerate}
\end{proposition}

\begin{proposition}\label{nonCM.numerical.2}
Let
$h^1(\mathcal O_Y(1))=0$ and let $\tilde \gamma$ be an integer.
\begin{enumerate}
 \item If
\begin{enumerate}
\item[(1.1)]$2 \leq \tilde \gamma \leq \textrm{min}(d-2,2d-r)$ when $g=0$; or
\item[(1.2)] $\frac{5g-1}{2} \leq \tilde \gamma \leq \textrm{min}(d,2d-r)$
when $g \geq 1$,
\end{enumerate}
then there exist
\begin{enumerate}
\item[(i)] locally non Cohen--Macaulay double structures $\widetilde Y$
supported
on $Y$ of arithmetic genus $\tilde \gamma$,
embedded in $\mathbf P^r$;
\item[(ii)] for some  $\hat
\gamma > \tilde \gamma$, ribbons
$\widehat Y$of arithmetic genus $\hat \gamma$  supported on $Y$   and contained
in $\widetilde Y$.
\end{enumerate}
\item If
\begin{enumerate}
\item[(2.1)] $2 \leq \tilde \gamma \leq \textrm{min}(d-1,2d-r)$ when $
0 \leq g \leq 1$; or
\item[(2.2)] $\frac{5g-1}{2} \leq \tilde \gamma \leq \textrm{min}(d,2d-r)$
when $g \geq 2$,
\end{enumerate}
then there
exist ribbons $\widetilde Y$ on $Y$ of
arithmetic genus $\tilde \gamma$, embedded in $\mathbf P^r$.
\end{enumerate}

\smallskip

\noindent Moreover, under hypothesis (1) or (2), there exist
a smooth irreducible algebraic curve $T$
with a distinguished closed point $0$, a flat family $\mathcal C$ over $T$ and
a
$T$--morphism
$\Phi: \mathcal C \longrightarrow \mathbf P^{\tilde g}_{T}$ such that
\begin{enumerate}
 \item[(a)] $\mathcal C_t$ is a smooth irreducible curve of genus $\tilde
\gamma$,
\item[(b)] $\Phi_t$ is an embedding
for all $t \neq 0$, and
\item[(c)] $\Phi_0$ is of degree $2$ onto $Y$ and $[\Phi(\mathcal
X)]_0=\widetilde Y$
\end{enumerate}
(in particular, $\widetilde Y$ is smoothable in $\mathbf
P^r$).
\end{proposition}

\begin{proof}
We prove (1) first. Let $\epsilon=\tilde \gamma-2g+1$.  The lower bounds on
$\tilde \gamma$ assumed in
(1.1) and (1.2) imply
that $2\epsilon \geq g  + 1$. Thus
we may choose a non trivial line bundle $\mathcal E$ on $Y$ of degree
$-\epsilon$, for which $\epsilon=e$ and $\tilde \gamma=\tilde g$, and
such that $|\mathcal E^{-2}|$ possesses a smooth, effective, non empty divisor,
so  (1)
of Corollary~\ref{general.nonCMsmoothing.ribbons} is satisfied.
By hypothesis, $h^1(\mathcal O_Y(1))=0$.
In addition, the upper bounds in
(1.1) and (1.2) imply $d \geq \tilde g$; this is equivalent to
$d - e \geq
2g -1$, so $h^1(\mathcal E \otimes \mathcal O_Y(1))=0$ and $\mathcal E$
satisfies (3)
of Corollary~\ref{general.nonCMsmoothing.ribbons}. On the other hand,
(1.1)  and (1.2) imply $\tilde g \leq 2d - r$, which is the same as
$\mathcal E$
satisfying hypothesis (2)
of Corollary~\ref{general.nonCMsmoothing.ribbons}.
Finally (1.1) and (1.2) also imply (1) and (2) of
Remark~\ref{general.nonCMsmoothing.remark}, so the thesis of (1) follows from
Proposition~\ref{general.nonCMsmoothing.corollary}, (1).

\smallskip

\noindent Now we prove (2). Let again $\epsilon=\tilde
\gamma-2g+1$. The lower bounds assumed
for $\tilde \gamma$ on (2.1) and (2.2) allow us (as (1.1) and (1.2) did before)
to choose a
non trivial line bundle $\mathcal E$ on $Y$ of degree $-\epsilon$, for which
$\epsilon=e$ and $\tilde \gamma=\tilde g$, and such that $|\mathcal E^{-2}|$
possesses a smooth, effective, non empty divisor
$B$. Thus  (1)
of Corollary~\ref{general.nonCMsmoothing.ribbons} is satisfied.  Moreover, (2.1)
and (2.2) imply that
$d -e \geq 0$ if $g =0$ and $d -e \geq g + 1$ otherwise, so
we may further assume $\mathcal E \otimes \mathcal O_Y(1)$ to be globally
generated, thus satisfying the hypothesis of
Remark~\ref{general.nonCMsmoothing.remark2}. Finally,  (2.1) and (2.2) also
imply  hypotheses (2) and (3) of Corollary~\ref{general.nonCMsmoothing.ribbons},
so
the thesis of (2) follows from
Proposition~\ref{general.nonCMsmoothing.corollary}, (2).
\end{proof}

\section{Double structures on rational normal curves}\label{rational}

\noindent In this section we study smoothable double structures supported on
rational normal curves. For this purpose we apply
Theorem~\ref{general.nonCMsmoothing} (or, rather,
Corollary~\ref{general.nonCMsmoothing.ribbons}) to obtain
Theorem~\ref{nonCMsmoothing}. The most well known example of smoothable
double structures on rational normal curves are canonical ribbons.
Among other things, Theorem~\ref{nonCMsmoothing} says that, in addition to
canonical ribbons, in $\mathbf P^3$ or in projective spaces of higher dimension
there exist smoothable double structures, both locally Cohen--Macaulay and  non
Cohen--Macaulay, supported on rational normal curves,
of any arithmetic genus and of any degree in the non special range. In this
particular case of $Y$ being a rational normal curve,
Theorem~\ref{nonCMsmoothing} is much stronger than
Proposition~\ref{nonCM.numerical.2}.

\begin{theorem}\label{nonCMsmoothing}
Let $\tilde \gamma$
be a non negative integer and let $Y$ be a smooth rational normal
curve of
degree $d$  in $\mathbf P^d$ (i.e, $d=r$ in this case), where either
$d=\tilde \gamma-1$ and $\tilde \gamma \geq 3$ or $d \geq
\textrm{max\,}(\tilde \gamma,3)$.
For any integer $\hat \gamma$ such that $\tilde \gamma \leq \hat \gamma \leq
d+1$ there exist
\begin{enumerate}
\item[(i)]  double structures $\widetilde Y$  supported
on $Y$ of arithmetic genus $\tilde \gamma$,
embedded in $\mathbf P^d$;
\item[(ii)] ribbons
$\widehat Y$, supported on $Y$  of arithmetic genus $\hat \gamma$, contained
in $\widetilde Y$;
\item[(iii)]
a smooth irreducible algebraic curve $T$
with a distinguished closed point $0$, a flat family $\mathcal C$ over $T$ and
a
$T$--morphism
$\Phi: \mathcal C \longrightarrow \mathbf P^{\tilde g}_{T}$ such that
\begin{enumerate}
\item[(a)] $\mathcal C_t$ is a smooth irreducible curve of genus $\tilde
\gamma$,
\item[(b)] $\Phi_t$ is an embedding
for all $t \neq 0$, and
\item[(c)] $\Phi_0$ is of degree $2$ onto $Y$ and
$[\Phi(\mathcal
X)]_0=\widetilde Y$
\end{enumerate}
(in particular, $\widetilde Y$ is smoothable in $\mathbf
P^r$).
\end{enumerate}
These double structures $\widetilde Y$ are locally Cohen--Macaulay if and only
if $\hat \gamma=\tilde \gamma$. If $\hat \gamma  > \tilde \gamma$, then
$\widetilde Y$ contains $\widehat Y$ and
some embedded points.
\end{theorem}

\begin{proof}
It is well known
that the conormal bundle of $Y$ inside $\mathbf P^d$ is
\begin{equation}\label{conormalPd}
\mathcal I/\mathcal I^2 \simeq \mathcal O_{\mathbf P^1}(-d-2)^{\oplus d-1}.
\end{equation}

\medskip
\noindent
If $d=\tilde \gamma -1$, then $\hat \gamma=\tilde \gamma \geq 3$ and $Y$ has
degree
$\tilde \gamma-1$ in $\mathbf P^{\tilde \gamma-1}$. Then
\begin{equation*}
 \mathcal I/\mathcal I^2 \simeq \mathcal O_{\mathbf P^1}(-\tilde
\gamma-1)^{\oplus \tilde \gamma -2}.
\end{equation*}
Thus there are double structures  $\widetilde Y$ on $Y$ with arithmetic genus
$\tilde \gamma$ and these are necessarily canonical ribbons. It is
well known (see e.g.~\cite{Fong}) that a canonical ribbon $\widetilde Y$ is the
flat limit of smooth canonical curves that approach the hyperelliptic locus and
that, if we set the family of morphisms $\Phi$ in (iii) to be the relative
canonical morphism of such family of curves, then $[\Phi(\mathcal
X)]_0=\widetilde Y$. This proves the result when $d=\tilde \gamma-1$.

\medskip
\noindent
Now assume $d \geq \tilde \gamma$, $d \geq 3$ and fix an integer $\hat \gamma$
such that $\tilde \gamma \leq \hat \gamma \leq
d+1$. Set $\mathcal E \simeq \mathcal O_{\mathbf P^1}(-\tilde \gamma-1)$ (thus
$e=\tilde \gamma+1$  and $\tilde g=\tilde \gamma$).
We claim that $\mathcal E$ satisfies (1), (2)  and (3) of
Corollary~\ref{general.nonCMsmoothing.ribbons}. Indeed, $\mathcal E^{-2}$ has
positive
degree, so $|\mathcal E^{-2}|$ possesses a smooth, effective, non empty
divisor, so (1)
of Corollary~\ref{general.nonCMsmoothing} is satisfied.
On the other hand,  $\mathcal O_Y(1)=\mathcal O_{\mathbf P^1}(d)$ and $\mathcal
E \otimes
\mathcal O_Y(1)=\mathcal O_{\mathbf P^1}(d-\tilde g-1)$, so (3) of
Corollary~\ref{general.nonCMsmoothing.ribbons} holds and
\begin{equation*}
h^0(\mathcal O_Y(1)) + h^0(\mathcal E \otimes \mathcal O_Y(1)) = 2d-\tilde g+1
\geq
d+1=r+1,
\end{equation*}
so (2) of Corollary~\ref{general.nonCMsmoothing.ribbons} also holds.
Finally, since $\hat \gamma \leq d+1$ and $d \geq 3$, there exist $d-1$ global
sections of
$\mathcal O_{\mathbf P^1}(d-\hat \gamma +1)$ which do not vanish simultaneously
at any
given point of $\mathbf P^1$. Then \eqref{conormalPd} yields the existence of a
nonzero homomorphism $\mu \in \textrm{Hom}(\mathcal I/\mathcal I^2,\mathcal E)$
whose image is the line subbundle $\mathcal O_{\mathbf P^1}(-\hat \gamma-1)$,
which we will call $\mathcal
E'$ (thus $\hat g=\hat \gamma$).
follows from Corollary~\ref{general.nonCMsmoothing.ribbons} and
Observation~\ref{nonCMremark}.
\end{proof}

\section{Double structures on elliptic normal curves}\label{elliptic.section}

\noindent In this section we apply Theorem~\ref{general.nonCMsmoothing} (or,
rather, Corollary~\ref{general.nonCMsmoothing.ribbons})  to
study smoothable  double structures on elliptic normal curves. Precisely, in
Theorems~\ref{elliptic} and \ref{elliptic2}  we show the existence of
smoothable double structures, both locally Cohen--Macaulay and non
Cohen--Macaulay, on elliptic normal curves in $\mathbf P^3$ or
$\mathbf P^4$. Theorems~\ref{elliptic} and \ref{elliptic2} do not follow from
Proposition~\ref{nonCM.numerical.2}.

\begin{theorem}\label{elliptic}
Let $Y$ be a smooth elliptic normal curve of degree $4$ in $\mathbf P^3$.
\begin{enumerate}
 \item If
$\mathcal E$ is a line bundle of degree $-4$ on $Y$ such that $\mathcal
E^{-1} \neq \mathcal O_Y(1)$,  then there are nonzero elements $\mu$ of
Hom$(\mathcal I/\mathcal I^2, \mathcal E)$.
For each one of these $\mu$, there
exist
\begin{enumerate}
\item[(i)] a double structure $\widetilde Y$  supported
on $Y$ of arithmetic genus
$5$,
embedded in $\mathbf P^3$;
\item[(ii)] a ribbon
$\widehat Y$, supported on $Y$  of arithmetic genus $\hat g$ and with
conormal
bundle $\mathcal E'$ ($\mathcal E'=\mathrm{im } \, \mu$), contained in
$\widetilde Y$; and
\item[(iii)]
a smooth irreducible algebraic curve $T$
with a distinguished closed point $0$, a flat family $\mathcal C$ over $T$ and
a
$T$--morphism
$\Phi: \mathcal C \longrightarrow \mathbf P^{3}_{T}$ such that
\begin{enumerate}
\item[(a)] $\mathcal C_t$ is a smooth irreducible curve of genus $5$,
\item[(b)] $\Phi_t$ is an embedding
for all $t \neq 0$, and
\item[(c)] $\Phi_0$ is of degree $2$ onto $Y$ and
$[\Phi(\mathcal
X)]_0=\widetilde Y$
\end{enumerate}
(in particular, $\widetilde Y$ is smoothable in $\mathbf
P^3$).
\end{enumerate}

\smallskip

\noindent Moreover, there are ribbons $\widehat Y$ as in (i), (ii), (iii), of
arithmetic genus
$5, 6, 7$ and $9$ and with conormal bundle $\mathcal E''$, for any line
bundle $\mathcal E''$ of degree $-5$ or $-6$ or satisfying $\mathcal
E''=\mathcal
E$ or satisfying
$\mathcal E'' = \mathcal O_Y(-2)$. Conversely, if $\widetilde Y$ and $\widehat
Y$ are as in (i), (ii), then $\hat g=5, 6, 7$ or $9$.

\smallskip

\noindent The double
structures $\widetilde Y$ above  are ribbons  if and
only
its arithmetic genus is $5$. In
contrast, if $\hat g =6, 7$ or $9$, then $\widetilde Y$ has
embedded points.

\smallskip

\item If
$\mathcal E$ is a line bundle on $Y$ such that $\mathcal
E^{-1} = \mathcal O_Y(1)$, then there are surjective elements $\mu$ of
Hom$(\mathcal I/\mathcal I^2, \mathcal E)$.
For each of these $\mu$ there
exist
\begin{enumerate}
\item[(iv)] a ribbon
$\widetilde Y$, supported on $Y$  of arithmetic genus $5$ and with
conormal
bundle $\mathcal E$; and
\item[(v)]
a smooth irreducible algebraic curve $T$
with a distinguished closed point $0$, a flat family $\mathcal C$ over $T$ and
a
$T$--morphism
$\Phi: \mathcal C \longrightarrow \mathbf P^{3}_{T}$ such that
\begin{enumerate}
\item[(a)] $\mathcal C_t$ is a smooth irreducible curve of genus $5$,
\item[(b)] $\Phi_t$ is an embedding
for all $t \neq 0$, and
\item[(c)] $\Phi_0$ is of degree $2$ onto $Y$ and
$[\Phi(\mathcal
X)]_0=\widetilde Y$
\end{enumerate}
(in particular, $\widetilde Y$ is smoothable in $\mathbf
P^3$).
\end{enumerate}

\end{enumerate}

\end{theorem}

\begin{proof}
First we see that there are ribbons $\widehat Y$ as in (i), (ii), of
arithmetic genus
$5, 6, 7$ and $9$ and with conormal bundle $\mathcal E''$, for any line
bundle $\mathcal E''$ of degree $-5$ or $-6$ or satisfying $\mathcal
E''=\mathcal
E$ or satisfying
$\mathcal E'' = \mathcal O_Y(-2)$. Note first that such a line bundle $\mathcal
E''$ is a subbundle of $\mathcal E$.
Then for our purpose it suffices to prove
the existence of an element $\mu$ of Hom$(\mathcal I/\mathcal I^2)$ with im
$\mu=\mathcal E''$.
Any smooth elliptic normal curve of degree $4$ in $\mathbf P^3$ is the complete
intersection of two quadrics, so $\mathcal N_{Y, \mathbf P^3} = \mathcal
O_Y(2) \oplus  \mathcal O_Y(2)$.
The line bundle $\mathcal O_Y(2) \otimes \mathcal E$ is of degree
$4$ on $Y$ so,
for any subbundle $\mathcal O_Y \otimes \mathcal E''$ of
$\mathcal O_Y(2) \otimes \mathcal E$ of degree $2, 3$ or $4$, there are
nonzero global sections $\sigma_1$ and $\sigma_2$ of $\mathcal O_Y(2) \otimes
\mathcal E''$ with disjoint zero loci. Then, for  any line bundle $\mathcal
E''$ of
degree $-5$ or $-6$ or satisfying $\mathcal E''=\mathcal E$ or
$\mathcal E'' = \mathcal O_Y(-2)$  there exists an element $\mu \in
\textrm{Hom}(\mathcal I/\mathcal I^2,\mathcal E)$ whose image is $\mathcal E''$.

\smallskip

\noindent Now we see that if $\widetilde Y$ and $\widehat
Y$ are as in (i), (ii), then $\hat g=5, 6, 7$ or $9$.
If $\mathcal E''$ is a subbundle of $\mathcal E$, then deg$\mathcal E'' \leq
-4$
(and if deg$\mathcal E'' = -4$, then  $\mathcal E=\mathcal E''$). If
deg$\mathcal E'' \leq -9$, then $\mathcal N_{Y,
\mathbf P^3} \otimes \mathcal E''$ does not have nonzero global sections, so if
$\mu$ is a nonzero element of Hom$(\mathcal I/\mathcal I^2, \mathcal E)$ its
image could not be such $\mathcal E''$.
On the other hand, if deg$\mathcal E'' = -7$, any two global sections of
$\mathcal O_Y \otimes \mathcal E''$ have a common zero, so there are no $\mu \in
\textrm{Hom}(\mathcal I/\mathcal I^2,\mathcal E)$ whose image is $\mathcal E''$.

\smallskip

\noindent Now assume $\mathcal E^{-1} \neq \mathcal O_Y(1)$. Obviously there
are nonzero elements $\mu$ in Hom$(\mathcal I/\mathcal I^2, \mathcal E)$
and, since (1),  (2) an
(3) of Corollary~\ref{general.nonCMsmoothing.ribbons} are satisfied, then (i),
(ii) and
(iii)
follow from Corollary~\ref{general.nonCMsmoothing.ribbons}. Moreover, the double
structures $\widetilde Y$ and the ribbons $\widehat Y$ whose existence we showed
in the first paragraph of this proof also satisfy (iii). Finally, last claim of
(1)
follows
from Observation~\ref{nonCMremark}.

\smallskip
\noindent
Now assume $\mathcal E^{-1}= \mathcal O_Y(1)$. Since one can choose two
sections  of $\mathcal O_Y(1)$ having no common zeros, there exist
surjective elements in $\mathrm{Hom}(\mathcal I/ \mathcal I^2,
\mathcal E)$.
Let
$\mu$ be one of them. Then $\mu$ corresponds to a ribbon on
$Y$ of arithmetic genus $5$, embedded in $\mathbf P^3$. Since in this case
$h^1(\mathcal E \otimes  \mathcal O_Y(1))=h^1(\mathcal O_Y) \neq 0$, we
cannot use Corollary~\ref{general.nonCMsmoothing.ribbons}. Instead we will
apply
\cite[Theorem
1.5]{criterion}. To do so, let $B$ be a smooth divisor in $|\mathcal E^{-2}|$,
let $\pi: C \longrightarrow Y$ be the double cover of $Y$ branched along $B$
and with trace zero module $\mathcal E$ and let $\varphi$ be the composition
of $\pi$ followed by the inclusion of $Y$ in $\mathbf P^3$. Now consider the
homomorphism $\Psi_2$ associated to $\varphi$ defined in
Proposition~\ref{morphism.miguel}.
By \eqref{sequence.miguel} and the fact that
$H^1(\mathcal N_\pi)=0$ (because $Y$ is a curve), $\Psi_2$ is surjective. Let
$\nu$ be a counterimage of $\mu$ and let $\tilde \varphi$ be the first order
infinitesimal deformation. To apply \cite[Theorem
1.5]{criterion} we observe that
$\varphi$ has an algebraic formally semiuniversal deformation because
$Y$ is a curve. Finally we need to check that $\varphi$ is
unobstructed. This holds if we show that $H^1(\mathcal N_{\varphi})=0$. Using
the sequence \eqref{sequence.miguel}, $H^1(\mathcal N_{\varphi})=0$ follows from
$H^1(\mathcal N_{Y, \mathbf P^3})=0,  H^1(\mathcal N_{Y, \mathbf P^3} \otimes
\mathcal E)=0$ and $H^1(\mathcal N_{\pi})=0$. This proves (2).
\end{proof}

\begin{theorem}\label{elliptic2}
Let $Y$ be a smooth elliptic normal curve of degree $5$ in $\mathbf P^4$.
\begin{enumerate}
 \item If
$\mathcal E$ is a line bundle of degree $-5$ on $Y$ such that $\mathcal
E^{-1} \neq \mathcal O_Y(1)$, then there are nonzero elements $\mu$ of
Hom$(\mathcal I/\mathcal I^2, \mathcal E)$.
For each of these $\mu$ there
exist
\begin{enumerate}
\item[(i)] a double structure $\widetilde Y$  supported
on $Y$ of arithmetic genus
$6$,
embedded in $\mathbf P^4$;
\item[(ii)] a ribbon
$\widehat Y$, supported on $Y$  of arithmetic genus $\hat g$ and with
conormal
bundle $\mathcal E'$, $(\mathcal E'= \textrm{im}\, \mu$) contained in
$\widetilde Y$; and
\item[(iii)]
a smooth irreducible algebraic curve $T$
with a distinguished closed point $0$, a flat family $\mathcal C$ over $T$ and
a
$T$--morphism
$\Phi: \mathcal C \longrightarrow \mathbf P^{4}_{T}$ such that
\begin{enumerate}
\item[(a)] $\mathcal C_t$ is a smooth irreducible curve of genus $6$,
\item[(b)] $\Phi_t$ is an embedding
for all $t \neq 0$, and
\item[(c)] $\Phi_0$ is of degree $2$ onto $Y$ and
$[\Phi(\mathcal
X)]_0=\widetilde Y$
\end{enumerate}
(in particular, $\widetilde Y$ is smoothable in $\mathbf
P^4$).
\end{enumerate}

\smallskip

\noindent Moreover, there are ribbons $\widehat Y$ as above of arithmetic genus
$\hat g=6, 7, 8$ and $9$ and with conormal bundle $\mathcal E''$, for any line
bundle $\mathcal E''$ of degree $-6, -7$ or $-8$ or such that $\mathcal
E''=\mathcal E$. Conversely, if $\widetilde Y$ and $\widehat
Y$ are as in (i), (ii), then $\hat g=6, 7, 8$ or $9$.

\smallskip

\noindent The double
structures $\widetilde Y$ above  are ribbons
if and only
its arithmetic genus is $6$. In
contrast, if $\hat g =7, 8$ or $9$, then $\widetilde Y$ has
embedded points.

\smallskip

\item If
$\mathcal E$ is a line bundle on $Y$ such that $\mathcal
E^{-1} = \mathcal O_Y(1)$,
then there are surjective elements $\mu$ of
Hom$(\mathcal I/\mathcal I^2, \mathcal E)$. For each one of these
$\mu$, there
exist
\begin{enumerate}
\item[(iv)] a ribbon
$\widetilde Y$, supported on $Y$  of arithmetic genus $6$ and with
conormal
bundle $\mathcal E$; and
\item[(v)]
a smooth irreducible algebraic curve $T$
with a distinguished closed point $0$, a flat family $\mathcal C$ over $T$ and
a
$T$--morphism
$\Phi: \mathcal C \longrightarrow \mathbf P^{4}_{T}$ such that
\begin{enumerate}
\item[(a)] $\mathcal C_t$ is a smooth irreducible curve of genus $6$,
\item[(b)] $\Phi_t$ is an embedding
for all $t \neq 0$, and
\item[(c)] $\Phi_0$ is of degree $2$ onto $Y$ and
$[\Phi(\mathcal
X)]_0=\widetilde Y$
\end{enumerate}
(in particular, $\widetilde Y$ is smoothable in $\mathbf
P^4$).
\end{enumerate}
\end{enumerate}
\end{theorem}

\begin{proof}
To prove both (1) and (2),
first we need to study Hom$(\mathcal I/\mathcal I^2,\mathcal E)$ for any line
bundle $\mathcal E$ on $Y$ of degree $-5$.
Let $O$ be a point of $Y$ such that $\mathcal O_Y(1)=\mathcal O_Y(5O)$.
Let $E^{\prime}_Y$
be the unique indecomposable rank $2$ vector bundle on $Y$ whose determinant is
$\mathcal
O_Y(O)$ and let
$E_Y$ be the rank $3$ vector bundle given by the unique non split extension
\begin{equation*}\label{n.split.s}
\xymatrix@R-20pt{
0 \ar[r] & \mathcal O_Y \ar[r] & E_Y \ar[r] &  E^{\prime}_Y \ar[r] & 0.
}
\end{equation*}
It follows from  \cite[Corollary V.1.4]{Hul} that
$\mathcal N_{Y, \mathbb P^4}= E_Y(8O)$.
Thus we have the exact sequences
\begin{equation}\label{n.split.normal}
\xymatrix@R-20pt{
0 \ar[r] & \mathcal O_Y(8O) \otimes \mathcal E \ar[r] & E^{\prime}_Y(8 O)
\otimes \mathcal E \ar[r] &
\mathcal O_Y(9 O) \otimes \mathcal E\ar[r] & 0,\\
0 \ar[r] & \mathcal O_Y(8O) \otimes \mathcal E \ar[r] & \mathcal N_{Y, \mathbb
P^4} \otimes \mathcal E \ar[r] &
 E^{\prime}_Y(8 O) \otimes \mathcal E \ar[r] & 0.
}
\end{equation}
Since $\mathcal E$ has degree $-5$, $\mathcal O_Y(8O) \otimes \mathcal E$ and $\mathcal O_Y(9O) \otimes \mathcal E$ are globally generated and $H^1(\mathcal O_Y(8O) \otimes \mathcal E)=0$. Then \eqref{n.split.normal} implies that
$\mathcal N_{Y, \mathbb P^4} \otimes \mathcal E$ is globally generated. Therefore $\mathcal N_{Y,
\mathbb P^4} \otimes \mathcal E$ has nowhere vanishing global sections, so
there are surjective elements $\mu$ of Hom$(\mathcal I/\mathcal I^2,\mathcal
E)$.

\smallskip
\noindent  Now we are going to prove (1) so we assume $\mathcal E^{-1} \neq
\mathcal O_Y(1)$. We have just seen
that there are $\mu$ which are surjective, so in particular, there are
homomorphisms $\mu$ in Hom$(\mathcal I/\mathcal I^2, \mathcal E)$
which are nonzero. Now we see that if $\widetilde Y$ and $\widehat
Y$ are as in (i), (ii), then $\hat g=5, 6, 7$ or $9$. For this we study the possible images of a nonzero element $\mu$ of Hom$(\mathcal I/\mathcal I^2, \mathcal E)$.
Recall
that in this case the image of $\mu$ is $\mathcal E'$, which is a line subbundle of
$\mathcal E$. Since the degree of $\mathcal E$ is $-5$, then the deg$\mathcal E'
\leq -5$. On the other hand, \eqref{n.split.normal} and the fact that $E'_Y$
does not split implies that if deg$\mathcal E' \leq -9$, then $\mathcal N_{Y,
\mathbb P^4} \otimes \mathcal E'$ does not have nonzero global sections, so
no such $\mathcal E'$ can be the image of $\mu$.

\smallskip
\noindent \noindent Now we see that there are ribbons $\widehat Y$ as in (i), (ii), of
arithmetic genus
$6, 7, 8$ and $9$ and with conormal bundle $\mathcal E''$, for any line
bundle $\mathcal E''$ of degree $-6, -7$ or $-8$ or satisfying $\mathcal
E''=\mathcal
E$.  For this it suffices to show the existence of
homomorphisms $\mu$ whose image is such a line bundle $\mathcal E''$. If $\mathcal E''=\mathcal E$,
we did already see this when we proved the existence of surjective elements in Hom$(\mathcal
I/\mathcal I^2, \mathcal E)$. If $\mathcal E''$ is of degree $-6,
-7$ or $-8$, then $\mathcal E''$ is
a subbundle of $\mathcal E$ so for each line bundle $\mathcal E''$ of
degree $-6, -7$ or $-8$ there is an effective divisor $D$ on $Y$, respectively
of degree $1, 2$ or $3$, such that $\mathcal E''=\mathcal E \otimes \mathcal O_Y(-D)$. Then,
such $\mathcal E''$ being the image of some
$\mu$ is equivalent to the existence of a nowhere vanishing global section of
$\mathcal N_{Y, \mathbb P^4} \otimes \mathcal E''$. For the latter, it suffices to show the
existence of a global section of $\mathcal
N_{Y, \mathbb P^4} \otimes \mathcal E$ vanishing exactly along $D$.  To prove
the existence of this global section we start arguing for  $\mathcal
E''=\mathcal O_Y(-8O)$. In this case $\mathcal O_Y(8O) \otimes \mathcal E=\mathcal O_Y(D)$, so there is global section of $\mathcal O_Y(8O) \otimes \mathcal E$ that maps to a nonzero global section of $\mathcal
N_{Y,
\mathbb P^4} \otimes \mathcal E$
vanishing exactly along
the divisor $D$.

\smallskip
\noindent Now suppose $\mathcal E'' \neq \mathcal O_Y(-8O)$ and is of degree
$-6,
-7$ or $-8$. Then it suffices to prove that
\begin{enumerate}
\item[(*)] there is a nowhere vanishing section
$t \in H^0(E'_Y(8 O) \otimes \mathcal E'')$.
\end{enumerate}

\smallskip
\noindent
Indeed, we consider the exact sequences
\begin{equation}\label{n.split.normal2}
\xymatrix@R-20pt{
0 \ar[r] & \mathcal O_Y(8O) \otimes \mathcal E'' \ar[r] & E'_Y(8 O)
\otimes \mathcal E'' \ar[r] &
\mathcal O_Y(9 O) \otimes \mathcal E''\ar[r] & 0,\\
0 \ar[r] & \mathcal O_Y(8O) \otimes \mathcal E'' \ar[r] & \mathcal N_{Y, \mathbb
P^4} \otimes \mathcal E'' \ar[r] &
 E^{\prime}_Y(8 O) \otimes \mathcal E'' \ar[r] & 0.
}
\end{equation}
Thus, if (*) holds, since $h^1(\mathcal O_Y(8O) \otimes \mathcal
E'')=0$, then
the section $t$ can be lifted to a section $s \in H^0(\mathcal
N_{Y, \mathbb P^4} \otimes \mathcal E'')$. Thus $s$ is  a
nowhere vanishing global section of $\mathcal
N_{Y, \mathbb P^4} \otimes \mathcal E''$ as wished.

\smallskip
\noindent Now let us prove (*).
If $\mathcal E''$ is a subbundle of $\mathcal E$ of degree $-6$, then
the first sequence of \eqref{n.split.normal2}
shows that $E^{\prime}_Y (8O) \otimes \mathcal E''$ is globally
generated. This implies (*) in this case.
If $\mathcal E''$  is a subbundle of $\mathcal E$ of degree $-7$, then
$\mathcal O_Y(8O) \otimes \mathcal E''=\mathcal O_Y(P)$ for some $P
\in Y$.
Let $t' \in H^0(E^{\prime}_Y(8 O) \otimes \mathcal E'')$ be a
counterimage of
a section $t'' \in H^0(\mathcal O_Y(9 O) \otimes \mathcal E'')$ not
vanishing at $P$.
Then, for a suitable $r \in H^0(\mathcal O_Y(8O) \otimes \mathcal E'')$
the zero
locus of $t=t'+r \in H^0(E^{\prime}_Y(8 O) \otimes \mathcal E'')$ is
empty.
Finally consider a subbundle $\mathcal E''$ of $\mathcal E$ of degree
$-8$ such that $\mathcal E'' \neq \mathcal
O(-8O)$.
In this case, $h^0(E^{\prime}_Y (8O) \otimes \mathcal E'')= h^0(\mathcal
O_Y(9O) \otimes \mathcal E'')=1$, so $E^{\prime}_Y (8O) \otimes \mathcal
E''$ has nonzero global sections which are either nowhere vanishing or
vanish exactly along one point of $Y$. We claim that only the
former happens. Suppose the contrary, i.e, suppose there exists a section $t
\in H^0(E^{\prime}_Y (8O) \otimes \mathcal E'')$ whose zero locus is exactly one
point of $Y$. Then, since $\mathcal E''$ is a degree $-8$
subbundle of $\mathcal E$, which has degree $-5$, $t$ gives rise to a global
section $m$ of $E^{\prime}_Y (8O) \otimes \mathcal E$ that vanishes on a
subscheme
of $Y$ of length exactly $4$. We see now that this is impossible. Indeed,
consider  the first sequence of \eqref{n.split.normal}. If $m$ maps to
$0 \in H^0(\mathcal O_Y(9O) \otimes \mathcal E)$, then $m$ comes from a
nonzero global section of $\mathcal O_Y(8O) \otimes \mathcal E$, which is a
line bundle of degree $3$, so $m$ would vanish at a subscheme of $Y$ of length
exactly $3$, not $4$. Thus $m$ should map to a nonzero
global section $m'$ of $\mathcal O_Y(9O) \otimes \mathcal E$. Since $\mathcal
O_Y(9O) \otimes \mathcal E$ is a line bundle of degree $4$, $m$ vanishes along a
subscheme of $Y$ of length $4$ if and only if the zero loci of $m$ and $m'$
are the same.
In that case the exact sequence \eqref{n.split.normal} will split, so $E'_Y$ would be decomposable and this is a contradiction.
Thus all the nonzero global sections
of $E^{\prime}_Y (8O) \otimes \mathcal E^{\prime}$ are nowhere vanishing.

\smallskip
\noindent Now we finish the proof of (1).
Since (1), (2) and (3) of  Corollary~\ref{general.nonCMsmoothing.ribbons} are
satisfied,
(i), (ii) and (iii) follow from Corollary~\ref{general.nonCMsmoothing.ribbons}
and the last claim of (1) follows from
Observation~\ref{nonCMremark}.

\medskip
\noindent Now we are going to prove (2), so assume $\mathcal E^{-1}=\mathcal
O_Y(1)$. As seen before, there are surjective homomorphisms in Hom$(\mathcal
I/\mathcal I^2,\mathcal E)$. In this case we cannot apply
Corollary~\ref{general.nonCMsmoothing.ribbons} because $\mathcal E
\otimes O_Y(1)$ equals $\mathcal O_Y$, which is special. Then (2) follows from
\cite[Theorem 1.5]{criterion} arguing as we did in the proof of
Theorem~\ref{elliptic}, (2).
\end{proof}

\section{Hom$(\mathcal I/\mathcal I^2,\mathcal E)$ and the semiuniversal
deformation of $\varphi$
}\label{geometric.section}

\noindent In Section~\ref{general.section}
we gave a geometric interpretation
of a given nonzero element $\mu$ of Hom$(\mathcal
I/\mathcal I^2, \mathcal E)$: $\mu$ ``produces'' a  multiple structure
$\widetilde
Y$ on $Y$ in $\mathbf P^r$ that appears when a family of embeddings (more
generally, a family of morphisms of degree $1$) degenerates  to a
double cover of $Y$ (in fact, in Section~\ref{general.section} we gave this
interpretation for homomorphisms $\mu$ whose image has rank $n$ and families of
embeddings degenerating to an $n$--to--one morphism, but recall that from
Section~\ref{general.double.section} onwards we are assuming $\mathcal E$ to be
a line bundle).  In this
section we explore further how this process takes
place. There are specific conditions under which the double structure
``produced'' by $\mu$ is determined uniquely. E.g., this happens if $\mu$ is
surjective (see Observation~\ref{nonCMremark}). However, in general, the
assignment of a double structure to an element $\mu$ of Hom$(\mathcal
I/\mathcal I^2, \mathcal E)$ is far from unique. Indeed, the next
Proposition~\ref{geom.inter.rat}  shows that $\widetilde Y$ is not uniquely
determined by $\mu$, not
even by the counterimage $\nu \in H^0(\mathcal N_\varphi)$ chosen in the
construction made in the proof of Theorem~\ref{general.nonCMsmoothing}.
Instead, $\mu$ only
determines
the minimal primary component of the ideal sheaf of $\widetilde Y$. To determine
$\widetilde Y$ completely we need to specify not only a tangent vector $v$ to
$\mathcal V$ but also the way in which we extend $v$ to an algebraic or, at
least, formal curve $T$, tangent to $v$, in the process to produce
a family of morphisms deforming $\varphi$. I.e, in order to determine
$\widetilde Y$ we need not only to look at a first order infinitesimal
deformation of $\varphi$ but also possibly at higher order infinitesimal
deformations of $\varphi$.

\begin{proposition}\label{geom.inter.rat}
Let $C$ be a smooth, irreducible hyperelliptic curve of genus $\tilde g$, let
$\pi: C \longrightarrow \mathbf P^1$ be its associated hyperelliptic double
cover and let $\varphi: C \longrightarrow \mathbf P^{\tilde g}$ be the morphism
induced by $H^0(\pi^*\mathcal O_{\mathbf P^1}(\tilde g))$. Let $\mathcal V$ be
the base of an algebraic formally
semiuniversal deformation
of $\varphi$ and let $0$ denote the point of
$\mathcal V$
corresponding to $\varphi$. Let $\mathcal U$ be the
locus
of $\mathcal V$
parameterizing embeddings from smooth irreducible curves to $\mathbf P^{\tilde
g}$,
let $\mathcal Z$ be the complement of $\mathcal U$ in $\mathcal V$ and
let $\mathcal H$ be the locus of $\mathcal V$
parameterizing morphisms of degree $2$ from smooth irreducible curves onto their
images in $\mathbf P^{\tilde g}$.
 There exists
\begin{enumerate}
 \item a
nonzero, non
surjective element   $\nu$ of $H^0(\mathcal N_\varphi)$;
\item two smooth irreducible algebraic curves $T_1$ and $T_2$, passing through
$0$ and tangent to $v$ (where $v$ is the vector of the
tangent space to $\mathcal V$ at $0$ corresponding to $\nu$), $T_1$
contained in $\mathcal Z \smallsetminus \mathcal H$ and $T_2$ contained
in $\mathcal U$ except for $0$; and
\item
two flat families $\mathcal C_1$ and $\mathcal
C_2$ of smooth irreducible curves of genus $\tilde g$ such that
the $T_1$--morphism
$\Phi_1: \mathcal C_1 \longrightarrow \mathbf P^{\tilde g}_{T_1}$ and the
$T_2$--morphism
$\Phi_2: \mathcal C_2 \longrightarrow \mathbf P^{\tilde g}_{T_2}$ obtained as
pullbacks from the semiuniversal deformation of $\varphi$ satisfy that
\smallskip
\begin{enumerate}
\item[(a)] $[\Phi_1(\mathcal
C_1)]_0$ is a ribbon of genus $\tilde g +1$ and, for all $t \in T_1, t
\neq 0$, $(\Phi_1)_t$ is of
degree $1$ and $(\Phi_1(\mathcal C_1))_t$ is a reduced and irreducible curve of
arithmetic genus $\tilde g +1$ with one singular point; and
\item[(b)] $[\Phi_2(\mathcal
C_2)]_0$ is a locally non Cohen--Macaulay double structure of arithmetic genus
$\tilde g$ containing a ribbon of arithmetic genus $\tilde g +1$ and an
embedded point and, for all $t \in T_2, t \neq 0$, $(\Phi_2)_t$ is an
embedding.
\end{enumerate}
\end{enumerate}
\end{proposition}

\begin{proof}
 Let $T_1$ be a smooth,
irreducible, algebraic curve with a distinguished closed point $0$ and let
$\mathcal C_1$ be a flat family over $T_1$ of smooth irreducible curves such
that
$(\mathcal C_1)_0=C$ and $(\mathcal C_1)_t$ is non hyperelliptic of genus
$\tilde g$ for
all $t \in T_1, t \neq 0$. Shrinking $T_1$ if necessary, we may consider a
relative line bundle $\mathcal L$ on $\mathcal C_1$ such that $\mathcal
L_0=\varphi^*\mathcal O_{\mathbf P^{\tilde g}}(1)$ and, for
all $t \in T_1$, $\mathcal L_t$ is
$\omega_{_{(\mathcal C_1)_t}}$ twisted by an effective line bundle of degree
$2$.
The relative global sections of $\mathcal L$ induce a $T_1$--morphism $\Phi_1:
\mathcal C_1 \longrightarrow \mathbf P^{\tilde g}_{T_1}$. Let $\Delta_1$ be the
first infinitesimal neighborhood of $0$ in $T_ 1$, let $\tilde
\varphi=(\Phi_1)_{\Delta_1}$, let $\nu$ be the element of  $H^0(\mathcal
N_\varphi)$ that corresponds to $\tilde \varphi$ and let $\mu=\Psi_2(\nu)$. In
addition, $\mathcal C_1$ can be taken so that
$\nu$, and therefore $\mu$, is nonzero.
Let $Y=\varphi(C)$, which is a smooth rational normal curve of degree $\tilde
g$. The image $\mathcal Y_1$ of $\Phi_1$ is a flat family of $1$--dimensional
subschemes of $\mathbf P^{\tilde g}$ such that, if $t \neq 0$, $(\mathcal
Y_1)_t$ is an irreducible and reduced singular curve of degree $2\tilde g$ and
arithmetic genus $\tilde g +1$, having exactly one singular point. Thus
$(\mathcal Y_1)_0$ is a double structure on $Y$ of arithmetic genus $\tilde g
+1$ that contains im$\tilde \varphi$.
Let $\mathcal E$ be the trace zero module of $\pi$. Then
$\mathcal E=\mathcal O_{\mathbf P^1}(-\tilde g-1)$ and the conormal bundle of
$Y$ in $\mathbf P^{\tilde g}$ is $\mathcal I/\mathcal I^2=\mathcal
O_{\mathbf P^1}(-\tilde g -2)^{\oplus \tilde g-1}$. Then the image of the
homomorphism  $\mu$  of $\textrm{Hom}(\mathcal I/\mathcal I^2,\mathcal E)$ is
either $\mathcal O_{\mathbf P^1}(-\tilde g -1)$ or  $\mathcal O_{\mathbf
P^1}(-\tilde g -2)$. By \cite[Theorem 3.8 (1)]{Gon}, in the first case im$\tilde
\varphi$ is a ribbon on $Y$ of arithmetic genus $\tilde g$ and this
would contradict
the fact that $(\mathcal Y_1)_0$ is a
double structure on $Y$ of arithmetic genus $\tilde g +1$. Thus necessarily the
image of $\mu$ is $\mathcal O_{\mathbf
P^1}(-\tilde g -2)$. Then im$\tilde
\varphi$ is a ribbon on $Y$ of arithmetic genus $\tilde g+1$ and  im$\tilde
\varphi=(\mathcal Y_1)_0$. This proves the existence of the desired family
$\Phi_1$ satisfying (a). Moreover, by formal semiuniversality, since $\nu$ is
nonzero, after possibly shrinking $T_1$ again and taking a suitable \'etale
cover, we may assume that $T_1$ admits an embedding to $\mathcal V$ that maps
$0$ to $[\varphi]$ and the tangent vector to $T_1$ at $0$ to $v$ in such a way
that the pullback of the semiuniversal deformation of $\varphi$ to $T_1$ is
$\Phi_1$. Then the image of $T_1 \smallsetminus \{0\}$ in $\mathcal V$ is
obviously contained in $\mathcal Z \smallsetminus \mathcal H$ by the
construction of $\Phi_1$ so, by an abuse of notation, we may identify $T_1$ with
its
image in $\mathcal V$.

\smallskip

\noindent Now, to construct $\Phi_2$ recall that, since $\mathcal O_Y(1)$ and
$\mathcal E \otimes \mathcal O_Y(1)$ are non special, by
\eqref{sequence.miguel} and   \eqref{N-F}, $\varphi$ is unobstructed. Then it is
possible to find a smooth, irreducible algebraic curve $T_2$ in $\mathcal V$,
passing through $0$, tangent to the vector $v$
corresponding to $\nu$ and contained in the locus $\mathcal U$
except for $0$. Pulling back the semiuniversal deformation of $\varphi$ to
$T_2$ we get a flat family $\mathcal
C_2$ of smooth irreducible curves of genus $\tilde g$ and a
$T_2$--morphism
$\Phi_2: \mathcal C_2 \longrightarrow \mathbf P^{\tilde g}_{T_2}$ such that
$(\Phi_2)_t$ is an embedding for all $t \in T_2, t \neq 0$ and
$(\Phi_2)_{\Delta_2}=\tilde \varphi$. Let $\mathcal Y_2=\Phi_2(\mathcal C_2)$.
Then $(\mathcal Y_2)_t$ is a smooth irreducible curve in $\mathbf P^{\tilde g}$
of degree $2\tilde g$ and genus $\tilde g$ and $(\mathcal Y_2)_0$ is a double
structure on $Y$ also of arithmetic genus $\tilde g$. As before, im$\tilde
\varphi$, which is a ribbon of arithmetic genus $\tilde g +1$, is contained in
$(\mathcal Y_2)_0$, so $(\mathcal Y_2)_0$ is the union of im$\tilde
\varphi$ and a double point supported on an embedded point of $(\mathcal
Y_2)_0$.
This proves the existence of a family $\Phi_2$ satisfying (b).
\end{proof}

\noindent Proposition~\ref{geom.inter.rat} provides an example that shows that, in
general, $\mu$ does not uniquely determine a double structure. However, we saw
(cf. Observation~\ref{nonCMremark}) that, if $\mu$ is surjective, the double
structure associated to $\mu$ is uniquely determined by $\mu$ (it is a ribbon
with conormal bundle $\mathcal E$). The next proposition
 tells the geometric reason behind this: if $\mu$ is surjective, any $\nu$ in
$H^0(\mathcal N_\varphi)$ lying over $\mu$ corresponds to a tangent vector to
$\mathcal V$ that only extends to paths contained in $\mathcal U$.

\begin{proposition}\label{geom.interp.2}
Let $\mu$ be a homomorphism of Hom$(\mathcal I/\mathcal I^2, \mathcal E)$.
Let $C$ be a
smooth irreducible curve of genus $\tilde g$
and let $\varphi: C \longrightarrow \mathbf P^r$ be a morphism from  $C$ to
$\mathbf P^r$ which factors through a double cover $\pi: C \longrightarrow Y$
with trace zero module $\mathcal E$.
Let $\mathcal V$ be the base of an algebraic formally
semiuniversal deformation
of $\varphi$ and let $[\varphi] \in \mathcal V$ be the point of $\mathcal V$
corresponding to $\varphi$.
Let $\nu$ be a counterimage of $\mu$ in $H^0(\mathcal
N_\varphi)$, let $\tilde \varphi$ be the first order infinitesimal deformation
corresponding to $\nu$  and let $v$ the tangent vector of
$\mathcal V$ corresponding to $\tilde \varphi$. Let $\mathcal U$ be the
locus
of $\mathcal V$
parameterizing embeddings from smooth irreducible curves to $\mathbf P^r$
and let $\mathcal Z$ be the complement of $\mathcal U$ in $\mathcal V$.
Suppose that $\varphi$ is unobstructed.
\begin{enumerate}
 \item $v$ is
tangent to an algebraic, smooth,
irreducible curve passing through $[\varphi]$
and contained in $\mathcal U$ except for $[\varphi]$; and
\item $v$ cannot be tangent
to any algebraic, smooth,
irreducible curve passing through $[\varphi]$
and contained in $\mathcal Z$.
\end{enumerate}
\end{proposition}

\begin{proof}
If $\varphi$ is unobstructed, then the base $\mathcal V$ of an
algebraic formally
semiuniversal deformation of $\varphi$ is smooth at
the point $[\varphi]$ corresponding to $\varphi$. By semiuniversality, there
exist a smooth algebraic irreducible curve $T$ with a distinguished point
$0$, a flat family $\mathcal C$ over $T$ and  a
$T$--morphism
$\Phi: \mathcal C \longrightarrow \mathbf P^{\tilde g}_{T}$ such that $\mathcal
C_t$ is a smooth irreducible curve, $\Phi_0=\varphi$ and
$\Phi_{_\Delta}=\tilde \varphi$.
\cite[Proposition
1.4]{criterion} implies that, after shrinking $T$ if necessary,
$\Phi_t$ is an embedding for all $t \in T, t \neq 0$ (since we are working with
$1$--dimensional schemes, there is an alternate argument for this, based on
studying what possible arithmetic genera the flat limit of
$\Phi_t(\mathcal C_t)$ might have). This proves (1).
Now suppose
there exists an algebraic smooth
irreducible curve $T$ passing through $[\varphi]$
and contained in $\mathcal Z$ whose tangent vector at $[\varphi]$ is $v$.
Then pulling back the semiuniversal deformation of $\varphi$ to $T$ we would
obtain a family of morphisms contradicting \cite[Proposition
1.4]{criterion}.
\end{proof}

\begin{remark}
 {\rm The unobstructedness of $\varphi$ required
in Proposition~\ref{geom.interp.2} is not a very strong condition. For
instance, $\varphi$ is unobstructed if both $\mathcal O_Y(1)$ and $\mathcal E
\otimes \mathcal O_Y(1)$ are nonspecial (this is condition (3) of
Theorem~\ref{general.nonCMsmoothing} and
Corollary~\ref{general.nonCMsmoothing.ribbons}). On the other hand,
Proposition~\ref{geom.interp.2} remains true if, instead of assuming $\varphi$
to be unobstructed, we assume the existence of
a smooth irreducible algebraic curve $T$
with a distinguished closed point $0$, a flat family $\mathcal C$ over $T$ and
a
$T$--morphism
$\Phi: \mathcal C \longrightarrow \mathbf P^{\tilde g}_{T}$
such that $\mathcal C_t$ is a smooth irreducible curve, $\Phi_0=\varphi$ and
$\Phi_{_\Delta}=\tilde \varphi$, where $\Delta$ is the first infinitesimal
neighborhood of $0$ in $T$.}
\end{remark}

\section*{Appendix}

\noindent We take advantage of this opportunity to fix a gap in the
article~\cite{GPcarpets}, written by the
first and third author. The gap concerns the arguments used
there to prove \cite[Theorem 3.5]{GPcarpets}.  \cite[Theorem 3.5]{GPcarpets}
says that $K3$ carpets supported on rational normal scrolls can be smoothed.
\cite[Theorem 3.5]{GPcarpets} is nevertheless true and a different,
independent proof of it was given in \cite[Corollary 2.9]{criterion}.

\medskip

\noindent We explain first what the problem is with the argument
in~\cite{GPcarpets} and
then we outline the way to fix it. Precisely,  the problem lies in \cite[Lemma
3.2]{GPcarpets}, which is  false as stated. The main thesis of \cite[Lemma
3.2]{GPcarpets} is this:

\smallskip

\noindent {\emph{Let  $\Cal X$ be a flat family of irreducible
varieties over
a smooth irreducible algebraic curve $T$ which is mapped to relative projective
space by a morphism
$\Phi$ induced by a relatively complete
linear series. Assume that $\Phi_t$ is an embedding for all $t \neq 0$ ($0 \in
T$)
and $\Phi_0$ is a finite morphism of degree $2$. Let $H$ be  a hyperplane in
projective space. Then $\Phi(\mathcal C) \cap (H \times T)$ is flat.}}

\smallskip

\noindent This claim is false in general. Indeed, if it were true, using
\cite[Theorem
2.1]{GPcarpets} and arguing
like in \cite[Corollary
3.3]{GPcarpets} we would show that the double structure $\widetilde Y$ that
appears in Corollary~\ref{general.nonCMsmoothing.ribbons} is always a ribbon and
this is
false as remarked in Observation~\ref{nonCMremark}.
The mistake in the proof of \cite[Lemma
3.2]{GPcarpets} is that, when we tensor the exact sequence
\begin{equation*}
0 \longrightarrow \mathcal O_{\mathcal Y}  \overset{\alpha}\longrightarrow
\Phi_*\mathcal O_{\mathcal C}  \longrightarrow \mathcal F  \longrightarrow 0
\end{equation*}
with $\mathcal O_{\mathbf P^n_T}/\mathcal I(H \times T)$, the resulting sequence
does not necessarily remain exact on the left.

\bigskip

\noindent We point out now how to avoid using \cite[Lemma 3.2]{GPcarpets} when
proving \cite[Theorem 3.5]{GPcarpets}. The proof of \cite[Theorem
3.5]{GPcarpets} is based on \cite[Proposition 3.4]{GPcarpets}, which says that
the
flat limit at $t=0$ of a family of $K3$ surfaces $\mathcal Y_t$ embedded by a
very ample polarization $\zeta_t$  is a $K3$ carpet provided that $(\mathcal
Y_0, \zeta_0)$ is a hyperelliptic polarized $K3$ surface. \cite[Proposition
3.4]{GPcarpets} was proved using \cite[Theorem 2.1]{GPcarpets} and \cite[Theorem
2.1]{GPcarpets}
essentially says the following:

\smallskip

\noindent \emph{A double structure $D$  of
dimension
$m$, supported on $D_{red}$ is locally Cohen--Macaulay if and only if through
every
closed point of $D$ there exists locally a Cartier divisor $h$ that cuts out on
$D$ a
locally Cohen--Macaulay double structure of dimension $m-1$, supported on the
restriction of $h$ to $D_{red}$.}

\smallskip
\noindent  It was in the process of applying
\cite[Theorem 2.1]{GPcarpets} to prove  \cite[Proposition
3.4]{GPcarpets} that we used \cite[Lemma 3.2]{GPcarpets}. Thus we outline now a
different
argument avoiding the use of \cite[Lemma 3.2]{GPcarpets}.
If the flat family $(\mathcal C,\zeta)$ over $T$ and the relative hyperplane
section $H \times T$ of \cite[Proposition 3.4]{GPcarpets} are suitably chosen,
then \cite[Theorem 2.1]{GPcarpets} can be applied to $\mathcal Y_0$ in the same
way as in the proof of \cite[Proposition 3.4]{GPcarpets}. Precisely, let $\tilde
\varphi$ be the restriction of $\Phi_\zeta$ to $H \times \Delta$. If one is able
to choose $(\mathcal C,\zeta)$ and the relative hyperplane section $H \times T$
so that $\tilde \varphi$ corresponds to a surjective homomorphism in
Hom$(\mathcal I/\mathcal I^2, \mathcal E)$, then  $\mathcal Y_0 \cap H$ will be
a canonical ribbon. Then \cite[Proposition 3.4]{GPcarpets} would follow from
\cite[Theorem 2.1]{GPcarpets}.


\begin{thebibliography}{GGP12}

\bibitem[Fon93]{Fong}
L.Y. Fong, \emph{Rational ribbons and deformation of hyperelliptic curves},
J. Algebraic Geom. {\bf 2} (1993),  295--307.


\bibitem[GGP08]{GGPropes} F.J. Gallego, M. Gonz\'alez, B.P. Purnaprajna, \emph{Deformation of finite morphisms and smoothing of ropes},  Compos. Math.  {\bf 144}  (2008),  673--688.




\bibitem[GGP12]{criterion} F.J. Gallego, M. Gonz\'alez, B.P. Purnaprajna,
\emph{An infinitesimal criterion to assure that a finite morphism can be
deformed to an embedding}, to appear in Revista Matem\'atica Complutense, doi
10.1007/s13163-011-0083-6.



\bibitem[GP97]{GPcarpets} F.J. Gallego, B.P. Purnaprajna, \emph{Degenerations of K3 surfaces in projective space},  Trans. Amer. Math. Soc.  {\bf 349}  (1997), 2477--2492.


\bibitem[Gon06]{Gon}
M. Gonz{\'a}lez, \emph{Smoothing of ribbons over curves}, J. reine angew. Math. {\bf 591} (2006), 201--235.




\bibitem[Hul86]{Hul}
K. Hulek, \emph{Projective geometry of elliptic curves}, Ast\'erisque {\bf 137}, Soci\'et\'e Math\'ematique de France, 1986.

\bibitem[HV85]{HV} K. Hulek, A. Van de Ven,
\emph{The Horrocks-Mumford bundle and the Ferrand construction},
Manuscripta Math. 50 (1985), 313--335.



\end{thebibliography}
\end{document}